\documentclass[12pt]{amsart}
\usepackage[margin=1in]{geometry}
\usepackage{amssymb, mathtools}
\usepackage[T1]{fontenc}
\usepackage{lmodern, microtype}
\usepackage[hidelinks]{hyperref}
\newtheorem{theorem}{Theorem}[section]
\newtheorem{lemma}[theorem]{Lemma}
\newtheorem{corollary}[theorem]{Corollary}
\newtheorem{proposition}[theorem]{Proposition}
\theoremstyle{definition}
\newtheorem{question}[theorem]{Question}
\theoremstyle{remark}
\newtheorem{remark}[theorem]{Remark}
\newtheorem{example}[theorem]{Example}

\numberwithin{equation}{section}

\DeclareMathOperator{\Lip}{Lip}
\DeclareMathOperator{\dist}{dist}
\DeclareMathOperator{\sign}{sign}
\newcommand{\R}{\mathbb{R}}
\newcommand{\pos}[1]{(#1)^{+}}
\newcommand{\Q}{\mathbf{Q}}
\newcommand{\G}{\mathcal{G}}
\newcommand{\lb}{[\![}
\newcommand{\rb}{]\!]}
\newcommand{\Vt}{\widetilde{V}}
\newcommand{\dt}{\widetilde{d}}
\newcommand{\T}{\mathcal T}
\newcommand{\Y}{\mathcal Y}
\newcommand{\one}{\mathbf 1}

\begin{document}

\title[Symmetric products and finite subset spaces of ALR]
{Symmetric products and finite subset spaces\\of absolute Lipschitz retracts}

\author{Leonid V. Kovalev}
\address{215 Carnegie, Department of Mathematics, Syracuse University,
Syracuse, NY 13244, USA}
\email{lvkovale@syr.edu}

\hypersetup{
  pdftitle={Symmetric products and finite subset spaces of absolute Lipschitz retracts},
  pdfauthor={Leonid V. Kovalev},
  pdfkeywords={metric space, symmetric product, absolute Lipschitz retract,
  multiple-valued map, finite subset space}
}

\begin{abstract}
We prove that the $Q$th symmetric product of an absolute
$\lambda$-Lipschitz retract, equipped with the bottleneck matching metric,
is an absolute $(3Q-2)\lambda$-Lipschitz retract. This answers a question
of De Pauw about the target spaces of Almgren's $Q$-valued maps. We also
prove that the space of nonempty subsets with at most $n$ points, equipped
with the Hausdorff metric, is an absolute $(6n-5)\lambda$-Lipschitz retract. The main
tool is a reconstruction theorem for bounded products of metric stars with
finitely many rays, modulo a finite permutation group: these quotients are absolute
$3$-Lipschitz retracts, independently of the group and the number of
coordinates. A variant of this theorem, in which unordered sets of sign
vectors take the place of the group action, handles finite subset spaces.
\end{abstract}

\maketitle

\section{Introduction}

Let $X$ be a metric space and let $Q\ge1$ be an integer. The symmetric group
$S_Q$ acts on the Cartesian power $X^Q$ by permuting the coordinates. The
quotient $\Q_Q(X)=X^Q/S_Q$ consists of unordered $Q$-tuples
$\lb x_1,\dots,x_Q\rb$ of points of $X$, in which repetitions are allowed. We
equip it with the metric
\begin{equation}\label{eq:metric}
\G\bigl(\lb x_1,\dots,x_Q\rb,\lb y_1,\dots,y_Q\rb\bigr)
=\min_{\sigma\in S_Q}\max_{1\le i\le Q} d(x_i,y_{\sigma(i)}).
\end{equation}
The space $\Q_Q(X)$ is the $Q$th \emph{symmetric product} of $X$. When
$X=\R^n$, the space $\Q_Q(\R^n)$ is the target of Almgren's multiple-valued
functions~\cite{Almgren,DeLellisSpadaro}; the notation $\Q_Q(X)$
follows~\cite{BouafiaDePauwGoblet,Goblet}.

A map $f\colon X\to Y$ between metric spaces is $L$-Lipschitz if
$d_Y(f(a),f(b))\le L\,d_X(a,b)$ for all $a,b\in X$; the least such $L$ is
denoted $\Lip(f)$. A metric space $Z$ is an \emph{absolute $\lambda$-Lipschitz
retract} if every metric space $M$ containing $Z$ admits a $\lambda$-Lipschitz
retraction $r\colon M\to Z$, that is, a $\lambda$-Lipschitz map such that
$r(z)=z$ for all $z\in Z$. If this holds for some $\lambda<\infty$, then $Z$ is
an \emph{absolute Lipschitz retract}.

T.~De Pauw asked whether $\Q_Q(X)$ is an absolute Lipschitz retract whenever
$X$ is~\cite[Problem 1.4]{AIMPL}. The same question, including the specific
cases $X=\ell_\infty$ and $X=C[0,1]$, appears in~\cite[Question
3.3.9]{BouafiaDePauwGoblet}. Our main result answers the question of De Pauw
in the affirmative.

\begin{theorem}\label{thm:main}
If $X$ is an absolute $\lambda$-Lipschitz retract, then $\Q_Q(X)$ is an
absolute $(3Q-2)\lambda$-Lipschitz retract.
\end{theorem}

Our metric~\eqref{eq:metric} is the bottleneck matching metric, denoted by
$\G_\infty$ in~\cite{BouafiaDePauwGoblet}; it is also the metric used by
Goblet~\cite{Goblet}. The root-sum-of-squares matching metric $\G_2$
satisfies $\G_\infty\le\G_2\le\sqrt Q\,\G_\infty$. Thus the absolute
Lipschitz retract property is the same for these two metrics, but all
numerical constants in this paper refer to $\G=\G_\infty$.

The space $\ell_\infty(I)$ of bounded real functions on a set $I$ is an
absolute $1$-Lipschitz retract, and so is every hyperconvex metric
space~\cite{AronszajnPanitchpakdi}. Every
absolute Lipschitz retract is a Lipschitz retract of $\ell_\infty(I)$ for some
$I$~\cite[Proposition 1.2]{BenyaminiLindenstrauss}. The spaces $c_0$ and
$C(K)$, with $K$ a compact metric space, are absolute Lipschitz
retracts as well~\cite{Lindenstrauss}; for $C(K)$, Kalton~\cite[Theorem
3.5]{Kalton} proved the bound $2$. Thus Theorem~\ref{thm:main} shows that
$\Q_Q(\ell_\infty(I))$ is an absolute $(3Q-2)$-Lipschitz retract for every
index set $I$ (compare~\cite[Question 2.4.4]{BouafiaDePauwGoblet}), and
that $\Q_Q(C(K))$ is an absolute $(6Q-4)$-Lipschitz retract.

For finite-dimensional targets, Almgren's bi-Lipschitz
embedding of $\Q_Q(\R^n)$ onto a Lipschitz retract of a Euclidean
space~\cite{Almgren}, \cite[Theorem 2.1]{DeLellisSpadaro} implies that
$\Q_Q(\R^n)$ is an absolute Lipschitz retract~\cite[Corollary
3.3.8]{BouafiaDePauwGoblet}. More generally, Goblet~\cite[Theorem 1.2]{Goblet}
used the extension theorems of Lang and Schlichenmaier~\cite{LangSchlichenmaier}
to prove that $\Q_Q(Y)$ is an absolute Lipschitz retract whenever $Y$ is a
complete weakly convex geodesic space of finite Nagata dimension, which
includes every finite-dimensional normed space. These results do not
provide dimension-independent constants.

A related sequence of spaces is formed by the finite subset spaces $X(n)$,
introduced by Borsuk and Ulam~\cite{BorsukUlam} under the name of
symmetric products; we follow Tuffley~\cite{Tuffley} in calling them finite
subset spaces. The elements of $X(n)$ are the nonempty subsets of $X$ of
cardinality at most $n$, with the Hausdorff distance
\[
d_H(A,B)=\max\Bigl(\max_{a\in A}\dist(a,B),\ \max_{b\in B}\dist(b,A)\Bigr).
\]
The question whether $X(n)$ is an absolute Lipschitz retract whenever $X$ is
was raised in~\cite{Kovalev2015,Kovalev2016} and further discussed
in~\cite{AkoforKovalev}. It is known that $\R(n)$ is an absolute Lipschitz
retract~\cite{Kovalev2015}, and so is $\R^d(n)$~\cite[Corollary
1.2]{Kovalev2016}. The map $\lb x,y\rb\mapsto\{x,y\}$ is an isometry of
$\Q_2(X)$ onto $X(2)$ (Lemma~\ref{lem:matching}), so Theorem~\ref{thm:main}
gives the following.

\begin{corollary}\label{cor:X2}
If $X$ is an absolute $\lambda$-Lipschitz retract, then $X(2)$ is an
absolute $4\lambda$-Lipschitz retract.
\end{corollary}

For $n\ge3$ the spaces $\Q_n(X)$ and $X(n)$ are different, since only the
former keeps track of multiplicities, and Theorem~\ref{thm:main} does not
yield $X(n)$ by a direct reduction (Section~\ref{sec:finite}). Nevertheless,
a variant of its proof answers the question for all $n$.

\begin{theorem}\label{thm:finite}
If $X$ is an absolute $\lambda$-Lipschitz retract, then $X(n)$ is an
absolute $(6n-5)\lambda$-Lipschitz retract.
\end{theorem}

For $n=2$, Corollary~\ref{cor:X2} gives the better constant $4\lambda$.
As above, Theorem~\ref{thm:finite} shows that $\ell_\infty(I)(n)$ and
$C(K)(n)$ are absolute Lipschitz retracts with the constants $6n-5$ and
$12n-10$, respectively. For $\ell_\infty(I)$ with $I\ne\varnothing$, the
constant must grow at least linearly in $n$
(Proposition~\ref{prop:finite-lower}).

The main step is a reconstruction theorem for products of metric stars. Let
$\T_{\mathcal C}$ be the metric star whose rays are indexed by a finite set
$\mathcal C$, that is, copies of $[0,\infty)$ indexed by $\mathcal C$ and glued
at $0$. For a set $\Gamma$, let $\ell_\infty(\Gamma,\T_{\mathcal C})$ be the
space of bounded families of points of $\T_{\mathcal C}$ indexed by $\Gamma$,
and let a subgroup $G\le\operatorname{Sym}(\mathcal C)$ act on it by applying
the same permutation of the rays in every coordinate. Theorem~\ref{thm:groups}
shows that $\ell_\infty(\Gamma,\T_{\mathcal C})/G$ is an absolute $3$-Lipschitz
retract for every set $\Gamma$.

Theorem~\ref{thm:main} reduces to the case $X=\ell_\infty(I)$. For an unordered
$Q$-tuple $\lb x_1,\dots,x_Q\rb$ in $\ell_\infty(I)$ and each $i\in I$, consider
the gaps between consecutive order statistics of the real numbers
$x_1(i),\dots,x_Q(i)$. Each gap is recorded as a point of the star whose rays
are indexed by the nonempty proper subsets of $\{1,\dots,Q\}$: its amplitude is
half the gap, and its color is the set of labels below the gap. Relabeling the
$Q$-tuple permutes all colors by the same element of $S_Q$. Together with the coordinatewise midrange, these
gaps determine the $Q$-tuple. The constant $3Q-2=1+3(Q-1)$ consists of $1$ from
the midrange and $3$ from each of the $Q-1$ gaps.

Theorem~\ref{thm:finite} also reduces to $X=\ell_\infty(I)$, but the gaps
can no longer be indexed by order statistics, because points of a finite set
may merge. Instead, each point $x$ of a finite set $A\subset\ell_\infty(I)$
is described, in each coordinate $i$, by a function of $t\in\R$: between the
extreme values of $\{a(i):a\in A\}$ it is the distance from $t$ to this set,
taken with the sign of $x(i)-t$, and outside them it vanishes. For a fixed
$A$ these functions differ only in sign, so together they form a finite set
of vertices of a rectangular box centered at the origin.
Theorem~\ref{thm:boxes} shows that such sets form an absolute
$3$-Lipschitz retract; its reconstruction records which sign patterns occur
rather than how many times they occur. The constant
$6n-5=1+3\cdot2(n-1)$ consists of $1$ from the midrange, $3$ from the
reconstruction, and $2(n-1)$ from recovering a point from its profile; the
last factor is sharp (Remark~\ref{rem:sharp}).

Section~\ref{sec:prelim} contains preliminaries, Section~\ref{sec:colors}
proves the reconstruction theorem, Section~\ref{sec:allQ} proves
Theorem~\ref{thm:main}, and Section~\ref{sec:finite} proves
Theorem~\ref{thm:finite}. Section~\ref{sec:questions} discusses lower bounds
and open questions.

\section{Preliminaries}\label{sec:prelim}

All normed spaces in this article are over the real field.

We will use the following characterization of absolute Lipschitz
retracts~\cite[Proposition 1.2]{BenyaminiLindenstrauss}: a metric space $Z$ is
an absolute $\lambda$-Lipschitz retract if and only if for every metric space
$M$, every nonempty subset $A\subset M$, and every $L$-Lipschitz map $f\colon A\to Z$
there exists an extension $F\colon M\to Z$ of $f$ with $\Lip(F)\le\lambda L$.
For $Z=\ell_\infty(I)$, an extension with
$\lambda=1$ is obtained by applying the McShane formula
$F_i(p)=\inf_{a\in A}\bigl(f_i(a)+L\,d(p,a)\bigr)$ 
to each coordinate. Fixing $a_0\in A$, we have
\[
f_i(a_0)-L\,d(p,a_0)\le F_i(p)
\le f_i(a_0)+L\,d(p,a_0),
\]
so $(F_i(p))_{i\in I}$ belongs to $\ell_\infty(I)$.

\begin{lemma}\label{lem:transfer}
Let $Z$ and $W$ be metric spaces, and let $\Phi\colon Z\to W$ and
$\Psi\colon W\to Z$ be Lipschitz maps such that $\Psi\circ\Phi=\mathrm{id}_Z$.
If $W$ is an absolute $\mu$-Lipschitz retract, then $Z$ is an absolute
$\lambda$-Lipschitz retract with $\lambda=\mu\Lip(\Phi)\Lip(\Psi)$.
\end{lemma}

\begin{proof}
Let $f\colon A\to Z$ be an $L$-Lipschitz map, where $A\subset M$. Extend
$\Phi\circ f$ to a map $F\colon M\to W$ with $\Lip(F)\le\mu\Lip(\Phi)L$. Then
$\Psi\circ F$ is an extension of $f$ with $\Lip(\Psi\circ F)\le\lambda L$.
\end{proof}

Note that $\Phi$ is not required to be an isometric embedding.

It is clear from~\eqref{eq:metric} that an $L$-Lipschitz map $g\colon X\to Y$
induces an $L$-Lipschitz map $\Q_Q(g)\colon\Q_Q(X)\to\Q_Q(Y)$, namely
$\lb x_1,\dots,x_Q\rb\mapsto\lb g(x_1),\dots,g(x_Q)\rb$. In particular, if $X$
is isometrically embedded in $Y$, then $\Q_Q(X)$ is isometrically embedded in
$\Q_Q(Y)$. Likewise, $A\mapsto g(A)$ is an $L$-Lipschitz map from $X(n)$ to
$Y(n)$ with respect to the Hausdorff distance.

\begin{lemma}\label{lem:matching}
For every metric space $X$, the map $\lb x,y\rb\mapsto\{x,y\}$ is an isometry
of $\Q_2(X)$ onto $X(2)$.
\end{lemma}

\begin{proof}
The map is a bijection, with repeated pairs corresponding to singletons.
Let $A=\lb a_1,a_2\rb$, $B=\lb b_1,b_2\rb$, and let $\rho$ be the
Hausdorff distance of their underlying sets. Clearly $\rho\le\G(A,B)$.
Relabel so that $d(a_1,b_1)\le\rho$. If $d(a_2,b_2)\le\rho$, this
matching proves the reverse inequality. Otherwise the Hausdorff bound
forces $d(a_2,b_1)\le\rho$ and $d(a_1,b_2)\le\rho$, so the crossed
matching does. No step uses $a_1\ne a_2$ or $b_1\ne b_2$.
\end{proof}

\begin{lemma}\label{lem:minima}
For two real families $(a_\gamma)_{\gamma\in F}$ and
$(b_\gamma)_{\gamma\in F}$ on a nonempty finite set $F$,
\[
\min_F a_\gamma+\min_F b_\gamma
\ge\min_F(a_\gamma+b_\gamma)
-\tfrac12\bigl(\max_F(a_\gamma-b_\gamma)-\min_F(a_\gamma-b_\gamma)\bigr).
\]
\end{lemma}

\begin{proof}
Write $a_\gamma=((a_\gamma+b_\gamma)+(a_\gamma-b_\gamma))/2$ and use the
minimum of each summand as a lower bound. Do the same for $b_\gamma$,
using the negative of the maximum of $a_\gamma-b_\gamma$, and add.
\end{proof}

\section{Reconstruction under finite permutation groups}\label{sec:colors}

Let $\mathcal C$ be a nonempty finite set of colors, and let
$G$ be a subgroup of $\operatorname{Sym}(\mathcal C)$. The metric star $\T_{\mathcal C}$
has origin $0$ and points $(a,r)$ with $a>0$ and $r\in\mathcal C$; we call $a$
the \emph{amplitude} and $r$ the \emph{color} of $(a,r)$. All symbols $(0,r)$
represent the origin, which has amplitude $0$ and no color. For $a,b\ge0$ and
$r,s\in\mathcal C$, the metric is
\[
d_{\T}((a,r),(b,s))=
\begin{cases}|a-b|,&r=s,\\a+b,&r\ne s.\end{cases}
\]
The group $G$ acts on $\T_{\mathcal C}$ by isometries via $g(a,r)=(a,g(r))$.

For a set $\Gamma$, the \emph{bounded product} $\ell_\infty(\Gamma,\T_{\mathcal C})$
consists of the families $u=(u_\gamma)_{\gamma\in\Gamma}$ of points of
$\T_{\mathcal C}$ with bounded amplitudes, equipped with the metric
$\sup_{\gamma\in\Gamma}d_{\T}(u_\gamma,v_\gamma)$. The group $G$ acts on it by
the same permutation in every coordinate: $(gu)_\gamma=g(u_\gamma)$. Write
\[
\Y_G(\Gamma)=\ell_\infty(\Gamma,\T_{\mathcal C})/G,\qquad
D([u],[v])=\min_{g\in G}\sup_{\gamma\in \Gamma}d_{\T}(u_\gamma,g(v_\gamma)).
\]
Since $G$ is finite and acts by isometries, $D$ is a metric on $\Y_G(\Gamma)$.
Thus a point of $\Y_G(\Gamma)$ is a bounded family of amplitudes, with a color
at each coordinate of positive amplitude, considered modulo one common element
of $G$.

\begin{theorem}\label{thm:groups}
For every nonempty finite color set $\mathcal C$, every $G\le\operatorname{Sym}(\mathcal C)$, and every set $\Gamma$, the space
$\Y_G(\Gamma)$ is an absolute $3$-Lipschitz retract.
\end{theorem}

We will obtain Theorem~\ref{thm:groups} from a more precise statement that specifies both an \textit{encoding} of 
$\Y_G(\Gamma)$ by elements of $\ell_\infty$ and its subsequent \textit{reconstruction}, both being Lipschitz maps.
A \emph{coloring} of a set $S\subset\Gamma$ is an element $r\in\mathcal C^S$. The group $G$ acts on colorings by
$(gr)_\gamma=g(r_\gamma)$, and $r|_H$ denotes the restriction of $r$ to a subset $H\subset S$.
We say that a \emph{test} $p$ consists of a nonempty finite subset $H_p\subset \Gamma$ and a function
$\sigma_p\colon \mathcal C^{H_p}\to\{-1,1\}$ that is constant on $G$-orbits. That is, $\sigma_p$ assigns a sign
to every coloring of $H_p$, and colorings related by $G$ receive the same sign. 
We say that a test $p$ is \emph{supported} in a set $S\subset\Gamma$ if $H_p\subset S$. 

A family $\mathcal P$ of tests is \emph{separating} if for every $S\subset \Gamma$ and all colorings $r,s\in\mathcal C^S$,
\begin{equation}\label{eq:separation}
\left.
\begin{gathered}
\sigma_p(r|_{H_p})=\sigma_p(s|_{H_p})\\
\text{for all }p\in\mathcal P\text{ supported in }S
\end{gathered}
\right\}
\quad\Longrightarrow\quad r=gs\text{ for some }g\in G.
\end{equation}
That is, the outcomes of the tests supported in $S$ determine every coloring of $S$ up to the action of $G$. 

\begin{example}\label{two-color-example} To illustrate the preceding definitions, consider the case of two colors: 
$\mathcal C=\{1,-1\}$ and $G=\operatorname{Sym}(\mathcal C)$. Identifying $(a,r)$ with $ar$ identifies $\T_{\mathcal C}$ with $\R$, and
therefore $\Y_G(\Gamma)$ is the quotient of $V=\ell_\infty(\Gamma)$ by the antipodal map $x\mapsto -x$.
We denote this \emph{antipodal quotient} by $\Vt$; its metric is 
\[
\dt([x],[y])=\min(\|x-y\|,\|x+y\|).
\]
By Theorem~\ref{thm:groups}, $\Vt$ is an absolute $3$-Lipschitz retract. 

We can also give an explicit separating family $\mathcal P$. For each two-element subset 
$H_p=\{\alpha, \beta\}\subset \Gamma$ and each coloring $r$ of $H_p$, let $\sigma_p(r) = r_\alpha r_\beta$; we index
this test by $\alpha\beta$. The values of such products on all two-element subsets of a set $S$ determine
$r\in\mathcal C^S$ up to a global change of sign. For this family, the encoding map $E$ of
Proposition~\ref{prop:reconstruction} below is
\[
E([x])_\alpha=|x_\alpha|,\qquad
E([x])_{\alpha\beta}=\sign(x_\alpha x_\beta)\min(|x_\alpha|,|x_\beta|),
\]
which may look discontinuous until one notices that 
\[
\sign(x_\alpha x_\beta)\min(|x_\alpha|,|x_\beta|) = \frac{|x_\alpha + x_\beta|}{2} - \frac{|x_\alpha - x_\beta|}{2},
\]
cf. \cite[Proposition 5.3.1]{Biermann}.
\end{example}

\begin{proposition}\label{prop:reconstruction}
Suppose $\mathcal P$ is a separating family of tests, and let
$J=\Gamma\sqcup\mathcal P$. There are maps
\[
E:\Y_G(\Gamma)\to\ell_\infty(J),\qquad
R:\ell_\infty(J)\to\Y_G(\Gamma)
\]
such that $E$ is $1$-Lipschitz, $R\circ E=\mathrm{id}$, and $\Lip(R)\le3$.
\end{proposition}

\begin{proof} \textsc{Step 1: encoding.}
Let $u$ be a representative of $[u]\in\Y_G(\Gamma)$, write $u_\gamma=(a_\gamma,r_\gamma)$, and define
\begin{equation}\label{eq:Egeneral}
E([u])_\gamma=a_\gamma,\qquad
E([u])_p=\sigma_p(r|_{H_p})\min_{\gamma\in H_p}a_\gamma.
\end{equation}
Since every $(0,r)$ is the origin, $r_\gamma$ is undefined when
$a_\gamma=0$; this causes no ambiguity, because then $E([u])_p=0$ for every
test $p$ with $\gamma\in H_p$, whatever the value of $\sigma_p$. The
invariance of $\sigma_p$ makes $E$ well defined.
To prove that $E$ is $1$-Lipschitz, let $\delta=D([u],[v])$, choose
representatives $u,v$ with $d_{\T}(u_\gamma,v_\gamma)\le\delta$ for all
$\gamma\in\Gamma$, and write $v_\gamma=(b_\gamma,s_\gamma)$.
The amplitudes differ by at most $\delta$, and so do their minima on any
finite set. If $\sigma_p$ takes different values on the two colorings,
some coordinate $\gamma\in H_p$ changes color. At this coordinate
$a_\gamma+b_\gamma\le\delta$, and hence the sum of the two minima is
at most $\delta$ as well. Thus $E$ is $1$-Lipschitz.

\textsc{Step 2: consistency.} 
For an arbitrary element $z\in\ell_\infty(J)$, its entries $z_p$, $p\in \mathcal P$, need not have any relation to $z_\gamma$, $\gamma\in \Gamma$.
We quantify this using the magnitude and sign of $z_p$. 

The quantity $\Delta(z,p)=\min_{\gamma\in H_p}z_\gamma-|z_p|$ is called the \emph{defect} of the test $p$ on the element $z$. A nonempty finite set $F\subset \Gamma$ is called \emph{inconsistent} for $z$ if there is no coloring $r\in\mathcal C^F$ satisfying
\begin{equation}\label{eq:constraints}
 \sigma_p(r|_{H_p})=\sign z_p   \quad\text{whenever $p$ is supported in $F$ and $z_p\ne0$;} 
\end{equation}
otherwise $F$ is \emph{consistent} for $z$.

We define a \textit{threshold} that will ensure that the truncated amplitudes admit a consistent reconstruction:
\begin{equation}\label{eq:threshold-general}
\tau(z)=\sup\Bigl( \{0\}\cup \bigl\{ \Delta(z,p) \colon p\in\mathcal P\bigr\} \cup\bigl\{\min_{\gamma\in F}z_\gamma:
F\text{ is inconsistent for }z\bigr\}\Bigr).
\end{equation}
Every term in~\eqref{eq:threshold-general} is at most
$\|z\|$, and $\tau(z)\ge0$. With $x^+ = \max(x,0)$, set
\[
a_\gamma=\pos{z_\gamma-\tau(z)},\quad S=\{\gamma\in \Gamma : a_\gamma>0\}.
\]
For every test $p$ supported in $S$, the inequality $\tau(z)\ge \Delta(z,p)$ gives 
\begin{equation}\label{eq:edge-general}
|z_p|\ge\min_{\gamma\in H_p}a_\gamma>0.
\end{equation}
Moreover, every finite $F\subset S$ is consistent: otherwise its contribution to $\tau(z)$ would be $\min_{\gamma\in F}z_\gamma>\tau(z)$.

\textsc{Step 3: Lipschitz bound.} We claim that the function $\tau$ defined by~\eqref{eq:threshold-general} is $2$-Lipschitz.
Indeed, let $\varepsilon=\|z-w\|$. Each test $p$ satisfies $|\Delta(z,p) - \Delta(w,p)| \le 2\varepsilon$. 
Suppose a set $F$ is inconsistent for $z$. If it is also inconsistent for $w$, its contributions to $\tau(z)$ and $\tau(w)$ 
differ by at most $\varepsilon$. Otherwise, choose a coloring on $F$ satisfying the constraints for $w$. It fails some
constraint for $z$, say at a test $p$ supported in $F$ with $z_p\ne0$. 
Consequently $z_pw_p\le 0$ which forces $|w_p|\le\varepsilon$. Using this together with $\Delta(w, p)\le \tau(w)$, we obtain 
\[
\min_{\gamma\in F}z_\gamma
\le\min_{\gamma\in H_p}w_\gamma+\varepsilon
\le\tau(w)+|w_p|+\varepsilon
\le\tau(w)+2\varepsilon.
\]
Taking suprema and interchanging $z,w$ proves the assertion.

\textsc{Step 4: reconstruction.} By~\eqref{eq:edge-general}, every test $p$ supported in $S$ has $z_p\ne0$. 
We claim that there is a coloring $r\in\mathcal C^S$ that satisfies the constraint~\eqref{eq:constraints} for each such $p$. If $S$ is finite, this is the consistency of $S$. In general, each constraint depends only on the coordinates in $H_p$, so it specifies a closed subset of the compact product $\mathcal C^S$, where $\mathcal C$ has the discrete topology. Finitely many constraints involve only a finite union $F$ of supports; since $F$ is consistent, they have a common solution on $F$, which extends arbitrarily to $S$. Compactness gives a simultaneous solution, and~\eqref{eq:separation} makes it unique modulo $G$. We may therefore define $R(z)\in \Y_G(\Gamma)$ by 
\begin{equation}\label{eq:Rgeneral}
R(z)=[(a_\gamma,r_\gamma)_{\gamma\in \Gamma}],
\end{equation}
with origin coordinates outside $S$.

Suppose that $z=E([u])$ for some $[u]\in \Y_G(\Gamma)$. Then $\Delta(z,p)\le 0$ for every $p$.  Moreover, $z_p\ne0$ forces all amplitudes of $u$ on $H_p$ to be
positive, so the colors of $u$, extended arbitrarily to the coordinates of zero
amplitude, satisfy~\eqref{eq:constraints} on every finite set. Hence no set is
inconsistent for $z$. Thus $\tau(z)=0$, the amplitudes are unchanged by $R$,
and $R(E([u]))=[u]$ by~\eqref{eq:separation}.

\textsc{Step 5: conclusion.}
That $E$ is $1$-Lipschitz and $RE=\mathrm{id}$ was shown above.
Let $\varepsilon=\|z-w\|$, and write
$R(z)=[(a_\gamma,r_\gamma)]$, $R(w)=[(b_\gamma,s_\gamma)]$.
The Lipschitz bound on $\tau$ gives
\begin{equation}\label{eq:amp}
|a_\gamma-b_\gamma|\le 3\varepsilon.
\end{equation}
On $S_0=\{\gamma:a_\gamma+b_\gamma>3\varepsilon\}$ both
amplitudes are positive: if one vanished, then~\eqref{eq:amp}
would give
$a_\gamma+b_\gamma=|a_\gamma-b_\gamma|\le3\varepsilon$.
Hence the terms $\tau(z)$ and $\tau(w)$ cancel in differences:
\[
|(a_\gamma-b_\gamma)-(a_{\gamma'}-b_{\gamma'})|
=|(z_\gamma-w_\gamma)-(z_{\gamma'}-w_{\gamma'})|\le2\varepsilon,
\qquad \gamma,\gamma'\in S_0.
\]
By Lemma~\ref{lem:minima}, every test supported in $S_0$ satisfies
\[
\min_{\gamma\in H_p}a_\gamma+\min_{\gamma\in H_p}b_\gamma
\ge\min_{\gamma\in H_p}(a_\gamma+b_\gamma)-\varepsilon
>2\varepsilon.
\]
Opposite signs of $z_p,w_p$ would contradict $|z_p-w_p|\le\varepsilon$
by~\eqref{eq:edge-general}. Thus
$\sigma_p(r|_{H_p})=\sign z_p=\sign w_p=\sigma_p(s|_{H_p})$ for every test
supported in $S_0$. Property~\eqref{eq:separation} supplies one $g\in G$ matching
the colors on $S_0$. After this permutation, the coordinate distances
on $S_0$ are at most $3\varepsilon$ by~\eqref{eq:amp}; outside $S_0$ they
are at most $a_\gamma+b_\gamma\le3\varepsilon$.
Therefore $D(R(z),R(w))\le 3\varepsilon$.
\end{proof}

\begin{proof}[Proof of Theorem~\ref{thm:groups}]
By Proposition~\ref{prop:reconstruction} and Lemma~\ref{lem:transfer}, it
suffices to find a separating family of tests. For
each nonempty finite $H\subset\Gamma$ and each $G$-orbit $O$ in
$\mathcal C^H$, take the test
\begin{equation}\label{eq:orbit-test}
\sigma_{H,O}(r)=
\begin{cases}1,&r\in O,\\-1,&r\notin O.\end{cases}
\end{equation}
If two colorings of $H$ lie in different orbits, the test for the orbit
of the first one takes the values $1$ and $-1$ on them. Hence
these orbit tests form a separating family by the following claim: \emph{if $r,s\in\mathcal C^S$
are not related by $G$, then there is a nonempty finite $H\subset S$
such that $r|_H$ and $s|_H$ are not related by $G$.} To prove it, choose
for each $g\in G$ a coordinate $\gamma_g\in S$ with
$gs_{\gamma_g}\ne r_{\gamma_g}$, which exists because $gs\ne r$. The set
$H=\{\gamma_g:g\in G\}$ is finite and nonempty, and no $g\in G$ carries
$s|_H$ to $r|_H$.
\end{proof}

\begin{remark}
In the setting of Example~\ref{two-color-example}, the
threshold~\eqref{eq:threshold-general} can be computed from pairs and triples
of distinct indices $\alpha,\beta,\gamma\in\Gamma$:
\[
\tau(z) = \sup\Bigl( \{0\}\cup \bigl\{\min(z_\alpha, z_\beta) - |z_{\alpha\beta}|\bigr\} \cup
\bigl\{\min(z_\alpha,z_\beta,z_\gamma) : z_{\alpha\beta} z_{\alpha\gamma} z_{\beta\gamma} < 0\bigr\}\Bigr).
\]
Indeed, a triple with $z_{\alpha\beta} z_{\alpha\gamma} z_{\beta\gamma} < 0$
is inconsistent for $z$. Conversely, let $F$ be inconsistent for $z$. If
$z_{\alpha\beta}=0$ for some $\alpha,\beta\in F$, the defect of the test
$\alpha\beta$ is $\min(z_\alpha,z_\beta)\ge\min_{\gamma\in F}z_\gamma$.
Otherwise every pair in $F$ carries a nonzero sign, and such an assignment of
signs to all pairs of $F$ is realized by a two-coloring as soon as it is
realized on every triple, so $F$ contains an inconsistent triple. In either
case, the contribution of $F$ to~\eqref{eq:threshold-general} is bounded by one
of the terms displayed above. After the
truncation $a_\gamma=\pos{z_\gamma-\tau(z)}$, every pair in $S$ carries a
nonzero sign by~\eqref{eq:edge-general}, and there are no triples with
$a_\alpha,a_\beta,a_\gamma>0$ and
$z_{\alpha\beta} z_{\alpha\gamma} z_{\beta\gamma} < 0$. This allows one to
find a two-coloring of $S$ that passes all tests supported in $S$.
\end{remark}

\section{Proof of Theorem~\ref{thm:main}}\label{sec:allQ}

We derive Theorem~\ref{thm:main} from its special case that is stated below as a proposition. 

\begin{proposition}\label{prop:allQ}
For every $Q\ge2$ and every index set $I$, the space
$\Q_Q(\ell_\infty(I))$ is an absolute $(3Q-2)$-Lipschitz retract.
\end{proposition}

\begin{proof}[Proof of Theorem~\ref{thm:main} and Corollary~\ref{cor:X2}]
The case $Q=1$ is immediate. For $Q\ge 2$, embed $X$ isometrically into
$V=\ell_\infty(I)$ for some suitable index set $I$ and choose a $\lambda$-Lipschitz retraction $\pi\colon V\to X$.
The induced map $\Q_Q(\pi)\colon\Q_Q(V)\to\Q_Q(X)$ is a $\lambda$-Lipschitz retraction. Proposition~\ref{prop:allQ} and
Lemma~\ref{lem:transfer} show that $\Q_Q(X)$ is an absolute Lipschitz retract with the constant $(3Q-2)\lambda$.
For $Q=2$, Lemma~\ref{lem:matching} gives the assertion about $X(2)$. 
\end{proof}

\begin{proof}[Proof of Proposition~\ref{prop:allQ}] 
Let $V=\ell_\infty(I)$. We will construct a $1$-Lipschitz embedding of $\Q_Q(V)$ into some $\ell_\infty$ space (an absolute $1$-Lipschitz retract) and then show that the embedding admits a $(3Q-2)$-Lipschitz left inverse.  

\textsc{Step 1: coordinatewise midrange.} 
For $A=\lb x_1,\ldots,x_Q\rb\in\Q_Q(V)$ and $i\in I$, let $x_k(i)$ denote the $i$th coordinate of $x_k$, and let 
$x_{(j)}(i)$ be the $j$th order statistic of $x_1(i),\dots,x_Q(i)$, that is, the $j$th smallest among these $Q$ numbers.
The order statistics do not depend on the labeling of $A$. If $B=\lb y_1,\ldots,y_Q\rb$ is labeled so that
$\|x_k-y_k\|\le\G(A,B)$ for every $k$, then, since order statistics are $1$-Lipschitz on $\R^Q$ with the maximum norm,
\begin{equation}\label{eq:orderstat}
|x_{(j)}(i)-y_{(j)}(i)|\le\G(A,B),\qquad i\in I,\quad j=1,\dots,Q.
\end{equation}
It follows that the \textit{coordinatewise midrange} $c(A)\in V$, defined by  
\[
c(A)_i=\frac{x_{(1)}(i)+x_{(Q)}(i)}{2},
\]
is a $1$-Lipschitz map from $\Q_Q(V)$ to $V$. 

\textsc{Step 2: encoding the gaps.} Let $\mathcal C=\{U\subset\{1,\ldots,Q\}:0<|U|<Q\}$. The symmetric group
$S_Q$ acts faithfully on $\mathcal C$ by permuting the underlying labels, so we regard it as a subgroup of
$\operatorname{Sym}(\mathcal C)$. Let $\Gamma=I\times\{1,\ldots,Q-1\}$, and consider the space $\Y_{S_Q}(\Gamma)$
of Section~\ref{sec:colors} for the color set $\mathcal C$. We write $t_{ij}$ for the coordinate $t_{(i,j)}$ of
$t\in\ell_\infty(\Gamma,\T_{\mathcal C})$.

Given $A=\lb x_1,\ldots,x_Q\rb\in\Q_Q(V)$ with a fixed labeling, define $t\in\ell_\infty(\Gamma,\T_{\mathcal C})$
as follows. For $(i,j)\in\Gamma$ let
\[
a_{ij}=\frac{x_{(j+1)}(i)-x_{(j)}(i)}{2}.
\]
If $a_{ij}=0$, let $t_{ij}$ be the origin of $\T_{\mathcal C}$. Otherwise let $t_{ij}=(a_{ij},U_{ij})$, where
\[
U_{ij}=\{k\in\{1,\dots,Q\} : x_k(i)\le x_{(j)}(i)\}.
\]
Since $x_{(j)}(i)<x_{(j+1)}(i)$ in this case, $U_{ij}$ consists of the labels of the $j$ smallest numbers among
$x_1(i),\dots,x_Q(i)$; in particular, $|U_{ij}|=j$ and $U_{ij}\in\mathcal C$. The amplitudes $a_{ij}$ are
bounded by $\max_k\|x_k\|$. Relabeling the elements of $A$ changes every color $U_{ij}$ by the same permutation,
so the orbit
$
T(A)=[t]\in\Y_{S_Q}(\Gamma)
$
is well defined.

\textsc{Step 3: Lipschitz bound.}
We claim that $T\colon \Q_Q(V) \to \Y_{S_Q}(\Gamma)$ is $1$-Lipschitz. Let $A,B\in\Q_Q(V)$, let $\delta=\G(A,B)$, and label
$A=\lb x_1,\ldots,x_Q\rb$, $B=\lb y_1,\ldots,y_Q\rb$ so that $\|x_k-y_k\|\le \delta$ for every $k$. Let $t$ and
$t'$ be the representatives of $T(A)$ and $T(B)$ obtained from these labelings, with amplitudes $a_{ij}$,
$a'_{ij}$ and colors $U_{ij}$, $U'_{ij}$. Since
$D(T(A),T(B))\le\sup_{(i,j)\in\Gamma}d_{\T}(t_{ij},t'_{ij})$, it suffices to show that
$d_{\T}(t_{ij},t'_{ij})\le\delta$ for every $(i,j)\in\Gamma$.

If $t_{ij}$ and $t'_{ij}$ lie on a common ray, which includes the case that one of them is the origin,
then~\eqref{eq:orderstat} gives
\[
d_{\T}(t_{ij},t'_{ij})=|a_{ij}-a'_{ij}|
\le\frac{|x_{(j+1)}(i)-y_{(j+1)}(i)|+|x_{(j)}(i)-y_{(j)}(i)|}{2}\le\delta.
\]
Otherwise both amplitudes are positive and $U_{ij}\ne U'_{ij}$. Since $|U_{ij}|=j=|U'_{ij}|$, we can choose
$k\in U_{ij}\setminus U'_{ij}$ and $l\in U'_{ij}\setminus U_{ij}$. Since $k\in U_{ij}$ and $l\notin U_{ij}$, we
have $x_k(i)\le x_{(j)}(i)$ and $x_l(i)\ge x_{(j+1)}(i)$, hence $2a_{ij}\le x_l(i)-x_k(i)$. Similarly,
$2a'_{ij}\le y_k(i)-y_l(i)$. Therefore
\[
d_{\T}(t_{ij},t'_{ij})=a_{ij}+a'_{ij}
\le\frac{\bigl(x_l(i)-y_l(i)\bigr)+\bigl(y_k(i)-x_k(i)\bigr)}{2}\le\delta,
\]
proving that $T$ is $1$-Lipschitz. 

\textsc{Step 4: reconstruction.} Let $\ell_\infty^Q$ denote $\R^Q$ with the maximum norm, on which $S_Q$ acts by
permuting coordinates. Let $\one=(1,\ldots,1)\in\R^Q$, and for $U\subset\{1,\dots,Q\}$ let $\one_U\in\R^Q$ be its
indicator vector. The map
\[
\eta:\T_{\mathcal C}\to\ell_\infty^Q,
\qquad \eta(a,U)=a(\one-2\,\one_U),\quad \eta(0)=0,
\]
is $S_Q$-equivariant and $1$-Lipschitz. On a common ray this follows by homogeneity, and on different rays by the
triangle inequality, since $\|\eta(a,U)\|=a$.
For $v\in V$ and $[t]\in\Y_{S_Q}(\Gamma)$, define $\Theta(v,[t])=\lb y_1,\ldots,y_Q\rb$, where
$y_1,\dots,y_Q\in V$ are given coordinatewise by
\begin{equation}\label{eq:HQ}
(y_1(i),\ldots,y_Q(i))=v_i\one+\sum_{j=1}^{Q-1}\eta(t_{ij}),\qquad i\in I.
\end{equation}
The vectors $y_k$ are bounded, and replacing $t$ by $gt$ with $g\in S_Q$ permutes $y_1,\dots,y_Q$ by the
equivariance of $\eta$; hence $\Theta$ is well defined. Choosing representatives $t,t'$ that realize the minimum in the definition
of $D([t],[t'])$ and using that $\eta$ is $1$-Lipschitz in each of the
$Q-1$ terms, we obtain
\begin{equation}\label{eq:HQlip}
\G(\Theta(v,[t]),\Theta(v',[t']))
\le\|v-v'\|+(Q-1)D([t],[t']).
\end{equation}

Moreover,
\begin{equation}\label{eq:HQinverse}
\Theta(c(A),T(A))=A.
\end{equation}
Indeed, fix a labeling $A=\lb x_1,\ldots,x_Q\rb$, let $t$ be the corresponding representative of $T(A)$ from
Step~2, and fix $i\in I$ and a label $k$. For each $j$ with $a_{ij}>0$ we have $k\notin U_{ij}$ exactly when
$x_k(i)\ge x_{(j+1)}(i)$. Summing the gaps between consecutive order statistics up to $x_k(i)$, and noting that
zero gaps contribute nothing, we obtain
\[
x_k(i)=x_{(1)}(i)+\sum_{j:\,a_{ij}>0,\ k\notin U_{ij}}2a_{ij},
\qquad
c(A)_i=x_{(1)}(i)+\sum_{j=1}^{Q-1}a_{ij}.
\]
Subtracting, we find that $x_k(i)-c(A)_i$ is the sum of $a_{ij}$ over those $j$ with $a_{ij}>0$ and
$k\notin U_{ij}$, minus the sum over those with $a_{ij}>0$ and $k\in U_{ij}$. This is the $k$th coordinate of
$\sum_{j}\eta(t_{ij})$. Hence $x_1,\dots,x_Q$ satisfy~\eqref{eq:HQ} with $v=c(A)$, which
proves~\eqref{eq:HQinverse}.

\textsc{Step 5: conclusion.} Let $\mathcal P$ be the family of orbit tests~\eqref{eq:orbit-test} for the action
of $S_Q$ on $\mathcal C$, which is separating by the proof of Theorem~\ref{thm:groups}, and let $E,R$ be the
maps of Proposition~\ref{prop:reconstruction}, with $J=\Gamma\sqcup\mathcal P$. Set
\[
\Phi(A)=(c(A),E(T(A)))\in V\oplus_\infty\ell_\infty(J),
\qquad \Psi(v,z)=\Theta(v,R(z)),
\]
where $V\oplus_\infty\ell_\infty(J)$ carries the norm $\max(\|v\|,\|z\|)$ and is thus isometric to
$\ell_\infty(I\sqcup J)$.
Since $c$, $T$, and $E$ are $1$-Lipschitz, $\Lip(\Phi)\le1$. Since $\Lip(R)\le3$, \eqref{eq:HQlip} gives
\[
\G(\Psi(v,z),\Psi(v',w))
\le\|v-v'\|+3(Q-1)\|z-w\|
\le(3Q-2)\max(\|v-v'\|,\|z-w\|).
\]
By~\eqref{eq:HQinverse} and $R\circ E=\mathrm{id}$, we have $\Psi\circ\Phi=\mathrm{id}$.
Since $\ell_\infty(I\sqcup J)$ is an absolute $1$-Lipschitz retract, Lemma~\ref{lem:transfer} proves the claim.
\end{proof}

\section{Finite subset spaces}\label{sec:finite}

For $n\ge3$, Theorem~\ref{thm:finite} cannot be obtained from
Theorem~\ref{thm:main} by Lemma~\ref{lem:transfer} with $\Psi$ the forgetful
map $\Q_n(X)\to X(n)$, $\lb x_1,\dots,x_n\rb\mapsto\{x_1,\dots,x_n\}$. This
map is $1$-Lipschitz, but it has no continuous right inverse if $X$ contains
two distinct non-isolated points $a$ and $b$. Indeed, a right inverse $s$ must
send every $n$-point set to the unordered $n$-tuple of its elements. Letting
$n-2$ additional points tend to $a$, we find that a continuous $s$ satisfies
$s(\{a,b\})=\lb a,\dots,a,b\rb$, with $a$ repeated $n-1$ times; letting them
tend to $b$ instead gives $s(\{a,b\})=\lb a,b,\dots,b\rb$. We therefore modify
the proofs of Sections~\ref{sec:colors} and~\ref{sec:allQ}, replacing the
group action by unordered sets of sign vectors, which record which sign
patterns occur but not how many times they occur.

Let $\Gamma$ be a set and $n\ge1$. A \emph{sign set} on a subset $S\subset\Gamma$
is a nonempty set $U\subset\{-1,1\}^S$ with $|U|\le n$; when $S=\varnothing$,
the only sign set is the singleton $\{-1,1\}^\varnothing$. For $H\subset S$ we
write $U|_H=\{u|_H:u\in U\}$, which is a sign set on $H$. A sign set on $S$ is
determined by its restrictions to finite subsets of $S$:
\begin{equation}\label{eq:finite-det}
U|_H=U'|_H\ \text{ for all nonempty finite } H\subset S
\quad\Longrightarrow\quad U=U'.
\end{equation}
Indeed, if $U\ne U'$, then after interchanging $U$ and $U'$ we may choose
$u\in U\setminus U'$. For each $v\in U'$ choose $\gamma_v\in S$ with
$u_{\gamma_v}\ne v_{\gamma_v}$. The set $H=\{\gamma_v:v\in U'\}$ is finite and
nonempty, and $u|_H\notin U'|_H$.

Let $\mathcal B_n(\Gamma)$ be the set of all nonempty sets
$A\subset\ell_\infty(\Gamma)$ with $|A|\le n$ for which there is a family
$a=(a_\gamma)_{\gamma\in\Gamma}$ of nonnegative numbers such that
\[
|u_\gamma|=a_\gamma\qquad(u\in A,\ \gamma\in\Gamma).
\]
Thus $A$ consists of vertices of a rectangular box centered at the origin,
possibly degenerate, and not every vertex needs to occur. We equip
$\mathcal B_n(\Gamma)\subset\ell_\infty(\Gamma)(n)$ with the Hausdorff distance
$d_H$. The family $a$ is the \emph{amplitude} of $A$. With
$S=\{\gamma:a_\gamma>0\}$, the \emph{sign set} of $A$ is
\[
U_A=\bigl\{(u_\gamma/a_\gamma)_{\gamma\in S}:u\in A\bigr\}.
\]
Conversely, a bounded family $a$ of nonnegative numbers and a sign set $U$ on
$S=\{\gamma:a_\gamma>0\}$ determine the element
$\{(a_\gamma u_\gamma)_{\gamma\in S}\oplus0:u\in U\}$ of $\mathcal B_n(\Gamma)$,
with $0$ in the coordinates outside $S$. As with the colors in
Section~\ref{sec:colors}, signs are defined only where the amplitude is
positive.

A \emph{test} $p$ consists of a nonempty finite set $H_p\subset\Gamma$ and a
sign vector $\theta_p\in\{-1,1\}^{H_p}$. For a sign set $U$ on a set
containing $H_p$, let $\sigma_p(U)=1$ if $\theta_p\in U|_{H_p}$, and
$\sigma_p(U)=-1$ otherwise. As before, $p$ is \emph{supported} in $S$ if
$H_p\subset S$. Let $\mathcal P$ be the set of all tests. For a finite set
$H$, the values of $\sigma_p(U)$ on the tests with $H_p=H$ determine $U|_H$.

\begin{theorem}\label{thm:boxes}
For every $n\ge1$ and every set $\Gamma$, the space $\mathcal B_n(\Gamma)$ is
an absolute $3$-Lipschitz retract. More precisely, let
$J=\Gamma\sqcup\mathcal P$. There are maps
\[
E:\mathcal B_n(\Gamma)\to\ell_\infty(J),\qquad
R:\ell_\infty(J)\to\mathcal B_n(\Gamma)
\]
such that $E$ is $1$-Lipschitz, $R\circ E=\mathrm{id}$, and $\Lip(R)\le3$.
\end{theorem}

\begin{proof}
The second statement implies the first by Lemma~\ref{lem:transfer}. We follow
the proof of Proposition~\ref{prop:reconstruction}, with sign sets in place of
colorings modulo $G$. Steps~2 and~3 carry over verbatim; we describe the
changes in Steps~1, 4, and~5.

\textsc{Step 1: encoding.} For $A\in\mathcal B_n(\Gamma)$ with amplitude $a$
and $S=\{\gamma:a_\gamma>0\}$, define
\[
E(A)_\gamma=a_\gamma,\qquad
E(A)_p=
\begin{cases}
\sigma_p(U_A)\min_{\gamma\in H_p}a_\gamma,& H_p\subset S,\\
0,&\text{otherwise}.
\end{cases}
\]
In both cases $|E(A)_p|=\min_{\gamma\in H_p}a_\gamma$. Let
$B\in\mathcal B_n(\Gamma)$ have amplitude $b$, and let $\delta=d_H(A,B)$.
For every $\gamma$, choosing $u\in A$ and $v\in B$ with $\|u-v\|\le\delta$
gives $|a_\gamma-b_\gamma|\le|u_\gamma-v_\gamma|\le\delta$. Hence
$\bigl||E(A)_p|-|E(B)_p|\bigr|\le\delta$, which proves
$|E(A)_p-E(B)_p|\le\delta$ unless $E(A)_p$ and $E(B)_p$ have opposite signs.
In that case $H_p$ lies in the supports of both amplitudes, and after
interchanging $A$ and $B$ we have $\theta_p\in U_A|_{H_p}$ and
$\theta_p\notin U_B|_{H_p}$. Choose $u\in A$ whose signs on $H_p$ are given
by $\theta_p$, and $v\in B$ with $\|u-v\|\le\delta$. The signs of $v$ on $H_p$
differ from $\theta_p$ at some $\beta\in H_p$, where $u_\beta$ and
$v_\beta$ have opposite signs. Therefore
\[
|E(A)_p-E(B)_p|=\min_{\gamma\in H_p}a_\gamma+\min_{\gamma\in H_p}b_\gamma
\le a_\beta+b_\beta=|u_\beta-v_\beta|\le\delta.
\]

\textsc{Steps 2 and 3: consistency and Lipschitz bound.} A nonempty finite set
$F\subset\Gamma$ is \emph{inconsistent} for $z\in\ell_\infty(J)$ if no sign
set $U$ on $F$ satisfies
\begin{equation}\label{eq:constraints-boxes}
\sigma_p(U)=\sign z_p\quad\text{whenever $p$ is supported in $F$ and
$z_p\ne0$.}
\end{equation}
With this definition, the defect $\Delta(z,p)$, the threshold $\tau(z)$, the
amplitudes $a_\gamma=\pos{z_\gamma-\tau(z)}$, and the set
$S=\{\gamma:a_\gamma>0\}$ are defined as in Step~2 of the proof of
Proposition~\ref{prop:reconstruction}. The arguments given there show
that~\eqref{eq:edge-general} holds for every test supported in $S$, that
every finite subset of $S$ is consistent, and, as in Step~3, that $\tau$ is
$2$-Lipschitz.

\textsc{Step 4: reconstruction.} By~\eqref{eq:edge-general}, $z_p\ne0$ for
every test $p$ supported in $S$. Represent a sign set on $S$ by an $n$-tuple
of elements of $\{-1,1\}^S$, with repetitions allowed; such $n$-tuples form
the compact space $\{-1,1\}^{\{1,\dots,n\}\times S}$. For a test $p$
supported in $S$, the constraint~\eqref{eq:constraints-boxes} on the set of
entries of an $n$-tuple depends only on the coordinates in
$\{1,\dots,n\}\times H_p$, and therefore defines a closed subset. Finitely
many such constraints involve only a finite union $F$ of supports. Since $F$
is consistent, they are satisfied by a sign set on $F$, which we represent by
an $n$-tuple whose entries we extend arbitrarily to $S$. By compactness, some
$n$-tuple satisfies all the constraints, and its set of entries $U$ is a sign
set on $S$. Since $z_p\ne0$ for every test supported in $S$, the
constraints~\eqref{eq:constraints-boxes} determine $U|_H$ for every finite
$H\subset S$, and~\eqref{eq:finite-det} shows that $U$ is unique. Let $R(z)$
be the element of $\mathcal B_n(\Gamma)$ with amplitude $a$ and sign set $U$.

Suppose that $z=E(A)$ for some $A\in\mathcal B_n(\Gamma)$. Then
$\Delta(z,p)\le0$ for every $p$. Moreover, $z_p\ne0$ forces $H_p$ into the
support of the amplitude of $A$. Hence, for every finite $F\subset\Gamma$,
the restrictions to $F$ of the elements of $U_A$, extended by $1$ outside
this support, form a sign set on $F$ that
satisfies~\eqref{eq:constraints-boxes}. Thus no set is inconsistent for $z$,
$\tau(z)=0$, the amplitudes are unchanged by $R$, and $R(E(A))=A$ by the
uniqueness of $U$.

\textsc{Step 5: conclusion.} Let $\varepsilon=\|z-w\|$, and let $a,U$ and
$b,U'$ be the amplitudes and sign sets of $R(z)$ and $R(w)$. As in Step~5 of
the proof of Proposition~\ref{prop:reconstruction}, we have
$|a_\gamma-b_\gamma|\le3\varepsilon$, and every test $p$ supported in
$S_0=\{\gamma:a_\gamma+b_\gamma>3\varepsilon\}$ satisfies
$\sigma_p(U)=\sign z_p=\sign w_p=\sigma_p(U')$. Hence $U|_H=U'|_H$ for every
finite $H\subset S_0$, and~\eqref{eq:finite-det} gives $U|_{S_0}=U'|_{S_0}$.
Thus every point of $R(z)$ has the same signs on $S_0$ as some point of
$R(w)$, and vice versa. For two such points, the coordinate differences are
at most $|a_\gamma-b_\gamma|\le3\varepsilon$ on $S_0$ and at most
$a_\gamma+b_\gamma\le3\varepsilon$ outside $S_0$. Therefore
$d_H(R(z),R(w))\le3\varepsilon$.
\end{proof}

\begin{remark}
The proofs of Proposition~\ref{prop:reconstruction} and
Theorem~\ref{thm:boxes} share the following structure. A point of
$\Y_G(\Gamma)$ is determined by its amplitudes and the $G$-orbit of its
coloring on the set where the amplitudes are positive, and a point of
$\mathcal B_n(\Gamma)$ is determined by its amplitude and its sign set. In
both cases, the distance between two points with amplitudes $a$ and $b$ is
the least $\delta\ge\sup_\gamma|a_\gamma-b_\gamma|$ such that the two
colorings modulo $G$, respectively the two sign sets, agree on
$\{\gamma:a_\gamma+b_\gamma>\delta\}$. For $D$ this follows from the
definitions of $D$ and $d_{\T}$. For $d_H$, two points at distance at most $\delta$
cannot have opposite signs at a coordinate where $a_\gamma+b_\gamma>\delta$;
conversely, if the sign sets agree on that set, then pairing points with
equal signs there shows that $d_H\le\delta$. Steps~1 and~5 of both proofs use
only this description of the metric.
\end{remark}

For a nonempty finite set $A\subset\R$, let $c(A)=(\min A+\max A)/2$ be its
midrange, and define
\[
h_A(t)=
\begin{cases}
\dist(t,A),&\min A\le t\le\max A,\\
0,&\text{otherwise},
\end{cases}
\]
\[
\phi_{A,x}(t)=\sign(x-t)\,h_A(t)\quad(x\in A,\ t\in\R).
\]
The value of $\sign0$ is immaterial, since $h_A(x)=0$. On each gap
$[s,s']$ between consecutive points of $A$, the function $h_A$ is the tent
$\min(t-s,s'-t)$, and $h_A=0$ outside these gaps; in particular, $h_A$ is
continuous and $\phi_{A,x}$ is bounded. We use the supremum norm for bounded
functions on $\R$.

\begin{lemma}\label{lem:profiles}
Let $A,B\subset\R$ be nonempty finite sets, $x\in A$, and $y\in B$.
\begin{enumerate}
\item[(a)] If $d_H(A,B)\le\varepsilon$ and $|x-y|\le\varepsilon$, then
$|c(A)-c(B)|\le\varepsilon$ and $\|\phi_{A,x}-\phi_{B,y}\|\le\varepsilon$.
\item[(b)] If $|A|,|B|\le n$, then
\begin{equation}\label{eq:scalar-inverse}
|x-y|\le|c(A)-c(B)|+2(n-1)\|\phi_{A,x}-\phi_{B,y}\|.
\end{equation}
\end{enumerate}
\end{lemma}

\begin{proof}
(a) The extreme points of $A$ and $B$ differ by at most $d_H(A,B)$, and so do
$\dist(t,A)$ and $\dist(t,B)$. Since
\[
h_A(t)=\min\bigl\{\dist(t,A),\pos{t-\min A},\pos{\max A-t}\bigr\},
\]
it follows that $|h_A(t)-h_B(t)|\le\varepsilon$ for all $t$. This gives
$|\phi_{A,x}(t)-\phi_{B,y}(t)|\le\varepsilon$ when $\sign(x-t)=\sign(y-t)$.
Otherwise $t$ lies in the closed interval between $x$ and $y$, and since
$h_A(t)\le|x-t|$ and $h_B(t)\le|y-t|$,
\[
|\phi_{A,x}(t)-\phi_{B,y}(t)|\le h_A(t)+h_B(t)\le|x-t|+|y-t|=|x-y|
\le\varepsilon.
\]

(b) Interchanging $(A,x)$ with $(B,y)$ if necessary, we may assume that
$x<y$. Let $\varepsilon=\|\phi_{A,x}-\phi_{B,y}\|$. For $x<t<y$ the values
$\phi_{A,x}(t)=-h_A(t)$ and $\phi_{B,y}(t)=h_B(t)$ have opposite signs, so
$h_A+h_B\le\varepsilon$ on $(x,y)$, and by continuity on $[x,y]$.

Let $m=\min(y,\max A)$. The points of $A$ divide $[x,m]$ into intervals,
each of which is contained in a gap $[s,s']$ of $A$ and has its left endpoint
at $s$. Since every such $s$ is a point of $A$ with $x\le s<\max A$, there
are at most $n-1$ intervals. Each of them has length at most
$2\varepsilon$: if $[s,s+L]\subset[s,s']$ and $h_A\le\varepsilon$ on
$[s,s+L]$, then either $L\le(s'-s)/2$ and $L=h_A(s+L)\le\varepsilon$, or
$(s'-s)/2=h_A((s+s')/2)\le\varepsilon$ and $L\le s'-s\le2\varepsilon$. Hence
$m-x\le2(n-1)\varepsilon$. The same argument applied to $B$, with intervals
measured from their right endpoints, shows that $m'=\max(x,\min B)$
satisfies $y-m'\le2(n-1)\varepsilon$.

If $\max A\ge y$, then $m=y$ and $y-x\le2(n-1)\varepsilon$; the same holds if
$\min B\le x$. Otherwise $m=\max A$ and $m'=\min B$. Since $\min A\le x$ and
$\max B\ge y$,
\[
2\bigl(c(B)-c(A)\bigr)\ge(\min B+y)-(x+\max A)
=2(y-x)-(m-x)-(y-m')\ge2(y-x)-4(n-1)\varepsilon,
\]
which proves~\eqref{eq:scalar-inverse}.
\end{proof}

\begin{remark}\label{rem:sharp}
For every $n\ge2$, the factor $2(n-1)$ in~\eqref{eq:scalar-inverse} cannot be reduced. Let
$A=\{0,2,\dots,2n-2\}$, $x=0$, $B=A+L$ with $L>2n-2$, and $y=L+2n-2$. The
functions $h_A$ and $h_B$ have disjoint supports and maximum $1$, so
$\|\phi_{A,x}-\phi_{B,y}\|=1$, while $|x-y|=|c(A)-c(B)|+2(n-1)$.
\end{remark}

\begin{proposition}\label{prop:finite}
For every $n\ge1$ and every index set $I$, the space $\ell_\infty(I)(n)$ is an
absolute $(6n-5)$-Lipschitz retract.
\end{proposition}

Before proving Proposition~\ref{prop:finite}, let us derive Theorem~\ref{thm:finite} from it.

\begin{proof}[Proof of Theorem~\ref{thm:finite}]
Embed $X$ isometrically into $V=\ell_\infty(I)$ for a suitable index set $I$,
and choose a $\lambda$-Lipschitz retraction $\pi\colon V\to X$. By
Section~\ref{sec:prelim}, $A\mapsto\pi(A)$ is a $\lambda$-Lipschitz
retraction of $V(n)$ onto $X(n)$. Proposition~\ref{prop:finite} and
Lemma~\ref{lem:transfer} complete the proof.
\end{proof}

\begin{proof}[Proof of Proposition~\ref{prop:finite}]
The cases $n=1$ and $I=\varnothing$ are immediate, so let $n\ge2$ and
$I\ne\varnothing$. Let $V=\ell_\infty(I)$ and $\Gamma=I\times\R$.

\textsc{Step 1: encoding.} For $A\in V(n)$ and $i\in I$, let
$A(i)=\{x(i):x\in A\}\subset\R$, where $x(i)$ is the $i$th coordinate of $x$.
The coordinatewise midrange $c(A)\in V$ is given by $c(A)_i=c(A(i))$. For
$x\in A$, define $u_{A,x}\in\ell_\infty(\Gamma)$ by
\[
u_{A,x}(i,t)=\phi_{A(i),x(i)}(t).
\]
The norms of $c(A)$ and $u_{A,x}$ are at most $\max_{x\in A}\|x\|$. Since
$|u_{A,x}(i,t)|=h_{A(i)}(t)$ does not depend on $x$, the set
$T(A)=\{u_{A,x}:x\in A\}$ belongs to $\mathcal B_n(\Gamma)$.

\textsc{Step 2: Lipschitz bound.} Let $A,B\in V(n)$ and $\delta=d_H(A,B)$.
Then $d_H(A(i),B(i))\le\delta$ for every $i$, so $\|c(A)-c(B)\|\le\delta$ by
Lemma~\ref{lem:profiles}(a). For $x\in A$ choose $y\in B$ with $\|x-y\|\le\delta$; Lemma~\ref{lem:profiles}(a), applied in each coordinate, gives $\|u_{A,x}-u_{B,y}\|\le\delta$. Thus every element of $T(A)$ lies within $\delta$ of $T(B)$. Interchanging $A$ and $B$ gives the reverse statement, so $d_H(T(A),T(B))\le\delta$. Thus $c$ and $T$ are $1$-Lipschitz.

\textsc{Step 3: decoding points.} Lemma~\ref{lem:profiles}(b), applied in
each coordinate, shows that
\begin{equation}\label{eq:point-decoder}
\|x-y\|\le\|c(A)-c(B)\|+2(n-1)\|u_{A,x}-u_{B,y}\|,
\qquad A,B\in V(n),\ x\in A,\ y\in B.
\end{equation}
Equip $V\times\ell_\infty(\Gamma)$ with the metric
$\rho((v,u),(v',u'))=\|v-v'\|+2(n-1)\|u-u'\|$. By~\eqref{eq:point-decoder},
the map $(c(A),u_{A,x})\mapsto x$ is well defined and $1$-Lipschitz on the
set $\{(c(A),u_{A,x}):A\in V(n),\ x\in A\}$. By the McShane formula
(Section~\ref{sec:prelim}), it extends to a $1$-Lipschitz map
$\Theta\colon(V\times\ell_\infty(\Gamma),\rho)\to V$.

\textsc{Step 4: conclusion.} Let $E,R$ be the maps of
Theorem~\ref{thm:boxes}, with $J=\Gamma\sqcup\mathcal P$, and set
\[
\Phi(A)=(c(A),E(T(A)))\in V\oplus_\infty\ell_\infty(J),
\qquad
\Psi(v,z)=\{\Theta(v,u):u\in R(z)\}\in V(n).
\]
Since $c$, $T$, and $E$ are $1$-Lipschitz, $\Lip(\Phi)\le1$. Given $(v,z)$ and
$(v',w)$, each $u\in R(z)$ has some $u'\in R(w)$ with
$\|u-u'\|\le d_H(R(z),R(w))\le3\|z-w\|$, and then
\[
\|\Theta(v,u)-\Theta(v',u')\|\le\|v-v'\|+6(n-1)\|z-w\|
\le(6n-5)\max(\|v-v'\|,\|z-w\|).
\]
Thus every point of $\Psi(v,z)$ lies within $(6n-5)\max(\|v-v'\|,\|z-w\|)$ of $\Psi(v',w)$. Interchanging $(v,z)$ and $(v',w)$ gives the reverse statement, so $\Lip(\Psi)\le6n-5$. Since $R\circ E=\mathrm{id}$,
\[
\Psi(\Phi(A))=\{\Theta(c(A),u_{A,x}):x\in A\}=A.
\]
Since $V\oplus_\infty\ell_\infty(J)$ is isometric to $\ell_\infty(I\sqcup J)$,
an absolute $1$-Lipschitz retract, Lemma~\ref{lem:transfer} proves the claim.
\end{proof}

\section{Lower bounds and open questions}\label{sec:questions}

Throughout this section, $V=\ell_\infty(I)$ and $\Vt$ is its antipodal
quotient from Example~\ref{two-color-example}, with the metric $\dt$. The
\emph{half-difference} map
\[
T_2\colon\Q_2(V)\to\Vt,\qquad T_2(\lb x,y\rb)=\Bigl[\frac{x-y}{2}\Bigr],
\]
is well defined and $1$-Lipschitz. Under the
identification of Example~\ref{two-color-example}, $T_2$ is the map $T$ of Step~2
in the proof of Proposition~\ref{prop:allQ} for $Q=2$.

\begin{proposition}\label{prop:lower}
Suppose $|I|\ge2$. Neither $\Vt$ nor $\Q_Q(V)$,
for any $Q\ge2$, is an absolute $\lambda$-Lipschitz retract with $\lambda<2$.
\end{proposition}

\begin{proof}
We use Biermann's configuration~\cite[Section 5.2]{Biermann}.
Fix two indices in $I$ and write the vectors supported on
them as pairs of real numbers. The points $P_1=[(0,3)]$, $P_2=[(2,1)]$, and
$P_3=[(-2,1)]$ of $\Vt$ are at distance $2$ from one another. We claim that
there is no $[p]\in\Vt$ with $\max_k\dt([p],P_k)<2$. Indeed, suppose
$\dt([p],P_k)\le r<2$ for $k=1,2,3$, and choose the representative $p$ so that
$\|p-(0,3)\|\le r$. Then $p_2\ge 3-r>1$, which implies $\|p+(2,1)\|>2$ and
$\|p+(-2,1)\|>2$. Therefore $\|p-(2,1)\|\le r$ and $\|p-(-2,1)\|\le r$, which
yield $2-r\le p_1\le r-2$, a contradiction.

Let $M=\{q_1,q_2,q_3,o\}$ be the metric space with $d(q_k,q_l)=2$ for $k\ne l$
and $d(o,q_k)=1$. The map $q_k\mapsto P_k$ is $1$-Lipschitz. A
$\lambda$-Lipschitz extension of this map to $M$ sends $o$ to a point within
distance $\lambda$ of each $P_k$, which by the above claim requires
$\lambda\ge2$.
By~\eqref{eq:metric}, the map $[b]\mapsto\lb b,-b\rb$ is an isometric
embedding of $\Vt$ into $\Q_2(V)$. Since $T_2(\lb b,-b\rb)=[b]$,
Lemma~\ref{lem:transfer} applied with $\Phi([b])=\lb b,-b\rb$ and $\Psi=T_2$
gives the same lower bound for $\Q_2(V)$.

For $Q\ge3$, put $\zeta=(10,0)$ and let $A_k$ be the unordered antipodal
pair corresponding to $P_k$. Pad each $A_k$ with $Q-2$ copies of $\zeta$
to obtain $B_k\in\Q_Q(V)$. Then $\G(B_k,B_l)=2$ for $k\ne l$:
matching the padding points gives the upper bound, and any matching
of cost less than $2$ would have to preserve the padding and would
contradict $\G(A_k,A_l)=2$.

Suppose a $Q$-tuple $B$ were within distance $r<2$ of all three $B_k$.
A matching to $B_1$ shows that exactly two of its entries have norm
less than $5$, while the other $Q-2$ have norm greater than $8$.
Since every point of every $A_k$ has norm at most $3$, every matching
of cost at most $r$ preserves this separation. Thus the two entries of
norm less than $5$ form a pair within distance $r$ of every $A_k$.
Applying $T_2$, which maps $A_k$ to $P_k$, gives a point of
$\Vt$ within distance $r$ of every $P_k$, a contradiction.
Using $B_k$ as the images of the points $q_k$ proves the lower bound for
every $Q\ge3$.
\end{proof}

The assumption $|I|\ge2$ in Proposition~\ref{prop:lower} cannot be dropped.
Indeed, sorting identifies $\Q_Q(\R)$ isometrically with the cone
$\{x\in\R^Q: x_1\le\dots\le x_Q\}$ in $\ell_\infty^Q$, and sorting is a
$1$-Lipschitz retraction of $\ell_\infty^Q$ onto this cone. Hence $\Q_Q(\R)$ is
an absolute $1$-Lipschitz retract, and so is the antipodal quotient of $\R$,
which is isometric to $[0,\infty)$.

\begin{question}
For $|I|\ge2$ and $Q\ge2$, what is the optimal absolute Lipschitz retract
constant of $\Q_Q(\ell_\infty(I))$? The bounds we have are
$2\le\lambda\le 3Q-2$. Is there an upper bound independent of $Q$? 
\end{question}

\begin{question} For $|I|\ge2$, what is the optimal absolute Lipschitz retract
constant of the antipodal quotient of $\ell_\infty(I)$? The bounds we have are $2\le\lambda\le3$, by 
Proposition~\ref{prop:lower} and Example~\ref{two-color-example}.
\end{question}

For finite subset spaces, the constant must grow with $n$, even when $I$
is a singleton.

\begin{proposition}\label{prop:finite-lower}
Let $I\ne\varnothing$ and $n\ge2$. If $V(n)$ is an absolute
$\lambda$-Lipschitz retract, then $\lambda\ge n-2$. If, in addition,
$|I|\ge2$, then $\lambda\ge2$.
\end{proposition}

\begin{proof}
Since $V(n)\subset V(n+1)$ isometrically, there is a $\lambda$-Lipschitz
retraction of $V(n+1)$ onto $V(n)$. As $V$ is a normed space of dimension at
least $1$, \cite[Theorem 1.3]{AkoforKovalev} gives $\lambda\ge(n+1)-3$.

Now let $|I|\ge2$, and write vectors supported on two fixed indices as pairs
of real numbers, as in the proof of Proposition~\ref{prop:lower}. The sets
\[
A_1=\{(0,3),(0,-3)\},\qquad A_2=\{(2,1),(-2,-1)\},\qquad
A_3=\{(-2,1),(2,-1)\}
\]
lie in $V(2)\subset V(n)$. By Lemma~\ref{lem:matching} and the isometric
embedding $[b]\mapsto\lb b,-b\rb$ of $\Vt$ into $\Q_2(V)$, we have
$d_H(A_k,A_l)=\dt(P_k,P_l)=2$ for $k\ne l$. Suppose that $B\in V(n)$ satisfies
$d_H(B,A_k)\le r<2$ for $k=1,2,3$. Some $c\in B$ satisfies
$\|c-(0,3)\|\le r$; writing $c_1,c_2$ for its two coordinates, we have
$c_2\ge3-r>1$. The point of $A_2$ within distance $r$ of $c$ must therefore
be $(2,1)$, which gives $c_1>0$, and the point of $A_3$ within distance $r$
of $c$ must be $(-2,1)$, which gives $c_1<0$. This contradiction shows that
no point of $V(n)$ is within distance less than $2$ of all three sets $A_k$.
Extending the map $q_k\mapsto A_k$ from the four-point space $M$ in the proof
of Proposition~\ref{prop:lower} now gives $\lambda\ge2$.
\end{proof}

The first bound in Proposition~\ref{prop:finite-lower} applies to $\R(n)$,
whereas $\Q_Q(\R)$ is an absolute $1$-Lipschitz retract for every $Q$. Thus
the constants for finite subset spaces must grow linearly in $n$ already on
the line, while for symmetric products no growth in $Q$ is known to be
necessary.

\begin{question}
For $|I|\ge2$ and $n\ge3$, what is the optimal absolute Lipschitz retract
constant of $V(n)$? By Theorem~\ref{thm:finite} and
Proposition~\ref{prop:finite-lower}, it lies between $\max(2,n-2)$ and
$6n-5$.
\end{question}

\section*{AI use disclosure}

The author acknowledges the use of AI tools, specifically ChatGPT by OpenAI and Claude by Anthropic. The text was edited and all mathematical arguments were checked by the author. 

\bibliographystyle{plain}
\bibliography{references}

\begin{thebibliography}{10}

\bibitem{AkoforKovalev}
Earnest Akofor and Leonid~V. Kovalev.
\newblock Growth rate of {L}ipschitz constants for retractions between finite subset spaces.
\newblock {\em Studia Math.}, 260(3):317--326, 2021.

\bibitem{Almgren}
Frederick~J. Almgren, Jr.
\newblock {\em Almgren's big regularity paper}, volume~1 of {\em World Scientific Monograph Series in Mathematics}.
\newblock World Scientific Publishing Co., Inc., River Edge, NJ, 2000.

\bibitem{AronszajnPanitchpakdi}
N.~Aronszajn and P.~Panitchpakdi.
\newblock Extension of uniformly continuous transformations and hyperconvex metric spaces.
\newblock {\em Pacific J. Math.}, 6:405--439, 1956.

\bibitem{BenyaminiLindenstrauss}
Yoav Benyamini and Joram Lindenstrauss.
\newblock {\em Geometric nonlinear functional analysis. {V}ol. 1}, volume~48 of {\em American Mathematical Society Colloquium Publications}.
\newblock American Mathematical Society, Providence, RI, 2000.

\bibitem{Biermann}
Patrick Biermann.
\newblock {\em Lipschitz {G}eometry of {B}anach and {M}etric {S}paces}.
\newblock ProQuest LLC, Ann Arbor, MI, 2016.
\newblock \url{https://surface.syr.edu/etd/421}.

\bibitem{BorsukUlam}
Karol Borsuk and Stanislaw Ulam.
\newblock On symmetric products of topological spaces.
\newblock {\em Bull. Amer. Math. Soc.}, 37(12):875--882, 1931.

\bibitem{BouafiaDePauwGoblet}
Philippe Bouafia, Thierry De~Pauw, and Jordan Goblet.
\newblock Existence of {$p$} harmonic multiple valued maps into a separable {H}ilbert space.
\newblock {\em Ann. Inst. Fourier (Grenoble)}, 65(2):763--833, 2015.

\bibitem{AIMPL}
Luca Capogna, Jeremy Tyson, and Stefan Wenger, editors.
\newblock {\em {AimPL}: {M}apping theory in metric spaces}.
\newblock American Institute of Mathematics, Palo Alto, 2012.
\newblock \url{http://aimpl.org/mappingmetric}.

\bibitem{DeLellisSpadaro}
Camillo De~Lellis and Emanuele~Nunzio Spadaro.
\newblock {$Q$}-valued functions revisited.
\newblock {\em Mem. Amer. Math. Soc.}, 211(991):vi+79, 2011.

\bibitem{Goblet}
Jordan Goblet.
\newblock Lipschitz extension of multiple {B}anach-valued functions in the sense of {A}lmgren.
\newblock {\em Houston J. Math.}, 35(1):223--231, 2009.

\bibitem{Kalton}
N.~J. Kalton.
\newblock Extending {L}ipschitz maps into {$C(K)$}-spaces.
\newblock {\em Israel J. Math.}, 162:275--315, 2007.

\bibitem{Kovalev2015}
Leonid~V. Kovalev.
\newblock Symmetric products of the line: embeddings and retractions.
\newblock {\em Proc. Amer. Math. Soc.}, 143(2):801--809, 2015.

\bibitem{Kovalev2016}
Leonid~V. Kovalev.
\newblock Lipschitz retraction of finite subsets of {H}ilbert spaces.
\newblock {\em Bull. Aust. Math. Soc.}, 93(1):146--151, 2016.

\bibitem{LangSchlichenmaier}
Urs Lang and Thilo Schlichenmaier.
\newblock Nagata dimension, quasisymmetric embeddings, and {L}ipschitz extensions.
\newblock {\em Int. Math. Res. Not.}, 58:3625--3655, 2005.

\bibitem{Lindenstrauss}
Joram Lindenstrauss.
\newblock On nonlinear projections in {B}anach spaces.
\newblock {\em Michigan Math. J.}, 11:263--287, 1964.

\bibitem{Tuffley}
Christopher Tuffley.
\newblock Finite subset spaces of {$S^1$}.
\newblock {\em Algebr. Geom. Topol.}, 2:1119--1145, 2002.

\end{thebibliography}

\end{document}